\documentclass{article}
\usepackage{xcolor}
\usepackage[a4paper,margin=1in]{geometry}
\usepackage{amsmath,amssymb,amsthm,mathtools}
\usepackage{enumitem}
\usepackage[colorlinks=true,linkcolor=blue,citecolor=blue,urlcolor=blue]{hyperref}
\newtheorem{theorem}{Theorem}[section]
\newtheorem{proposition}[theorem]{Proposition}
\newtheorem{lemma}[theorem]{Lemma}

\newtheorem{conjecture}[theorem]{Conjecture}

\newcommand{\OO}{\mathcal{O}}
\title{Sign law for Ramanujan's third order mock theta function $\rho(q)$}
\author{
Manosij Ghosh Dastidar\\
Institute of Discrete Mathematics and Geometry\\
TU Wien\\
}

\begin{document}

\maketitle

\begin{abstract}
We study the coefficients of Ramanujan's third order mock theta function
\[
\rho(q)=\sum_{m\geq 0}
\frac{q^{2m(m+1)}}{(1+q+q^2)(1+q^3+q^6)\cdots(1+q^{2m+1}+q^{4m+2})}
=\sum_{n\geq 0}r(n)q^n.
\]
Numerical evidence suggests the striking sign pattern
\[
r(3n)>0,\qquad r(3n+1)\leq 0,\qquad r(3n+2)\leq 0.
\]
We prove an asymptotic form of this phenomenon. More precisely, using Watson's relation between $\rho(q)$ and $\omega(q)$, together with a Rademacher-type expansion for the coefficients of $\omega(q)$ and the corresponding expansion for a theta--eta product, we show that
\[
r(n)\sim \kappa_{n\bmod 3}\,
\frac{2\pi}{(12n+8)^{1/4}}
I_{1/2}\!\left(\frac{\pi\sqrt{12n+8}}{18}\right),
\]
where
\[
\kappa_0=\frac13\cos\frac{\pi}{18}>0,
\qquad
\kappa_1=-\frac13\sin\frac{2\pi}{9}<0,
\qquad
\kappa_2=-\frac13\sin\frac{\pi}{9}<0.
\]
Consequently,
\[
r(3n)>0,
\qquad
r(3n+1)<0,
\qquad
r(3n+2)<0
\]
for all sufficiently large $n$. 
\end{abstract}

\section{Introduction}

The mock theta functions were first introduced to the world by the last letter Ramanujan wrote to Hardy from his deathbed. Since then to the present day, mock theta functions have occupied a central place in the theory of $q$-series, partitions, harmonic Maass forms etc; see, for instance, \cite{RamanujanLostNotebook1988,Watson1936,Zwegers2002,Zagier2009,GordonMcIntosh2012}. Among the third order mock theta functions is 
\begin{equation}\label{eq:rho-def}
\rho(q)=\sum_{m\geq 0}
\frac{q^{2m(m+1)}}{(1+q+q^2)(1+q^3+q^6)\cdots(1+q^{2m+1}+q^{4m+2})}.
\end{equation}
Write
\[
\rho(q)=\sum_{n\geq 0}r(n)q^n.
\]
The first few coefficients are
\[
1,-1,0,1,0,-1,1,-1,0,1,-1,0,2,-1,-1,1,-1,-1,2,-1,0,\ldots.
\]
These coefficients suggest a surprisingly rigid sign law:
\[
r(n)>0\quad(n\equiv 0\pmod 3),
\qquad
r(n)<0\quad(n\equiv 1,2\pmod 3),
\]
apart from a small number of early zeros in the negative residue classes.

The purpose of this note is to prove the asymptotic version of this sign law. Our proof belongs to the circle of ideas surrounding asymptotic and exact formulas for mock theta coefficients, beginning with work of Dragonette and Andrews and continuing through the harmonic-Maass-form approach of Bringmann and Ono \cite{Dragonette1952,Andrews1966,BringmannOno2006,BringmannOno2010}. The exact finite cut-off remains to be determined separately. Experimentally, the only zeros in the negative residue classes appear to be
\[
r(2)=r(4)=r(8)=r(11)=r(20)=0.
\]
Thus the natural finite conjecture is the following.

\begin{conjecture}\label{conj:finite-rho}
Let $\rho(q)=\sum_{n\geq 0}r(n)q^n$ be defined by \eqref{eq:rho-def}. Then
\[
r(3n)>0\qquad(n\geq 0),
\]
\[
r(3n+1)\leq 0\qquad(n\geq 0),
\]
and
\[
r(3n+2)\leq 0\qquad(n\geq 0).
\]
Moreover the only zeros in the last two families are
\[
r(2)=r(4)=r(8)=r(11)=r(20)=0.
\]
\end{conjecture}

Our main theorem proves the eventual part of this conjecture.

\begin{theorem}\label{thm:eventual-sign}
Let $\rho(q)=\sum_{n\geq 0}r(n)q^n$. Then, for all sufficiently large $n$,
\[
r(3n)>0,
\qquad
r(3n+1)<0,
\qquad
r(3n+2)<0.
\]
More precisely, as $n\to\infty$,
\begin{equation}\label{eq:rho-main-asymptotic}
r(n)
=
\kappa_{n\bmod 3}
\frac{2\pi}{(12n+8)^{1/4}}
I_{1/2}\!\left(\frac{\pi\sqrt{12n+8}}{18}\right)
+
\OO\!\left(e^{\pi\sqrt{12n+8}/24}\right),
\end{equation}
where
\[
\kappa_0=\frac13\cos\frac{\pi}{18},
\qquad
\kappa_1=-\frac13\sin\frac{2\pi}{9},
\qquad
\kappa_2=-\frac13\sin\frac{\pi}{9}.
\]
\end{theorem}

The proof uses Watson's identity \cite{Watson1936}
\[
2\rho(q)+\omega(q)=T(q)
\]
to reduce the asymptotic sign law to a comparison of the first non-cancelling Rademacher terms of $\omega(q)$ and of a theta--eta product $T(q)$.

\section{Watson's identity}

We now use a classical relation of Watson \cite{Watson1936}. Let
\[
\omega(q)=\sum_{m\geq 0}
\frac{q^{2m(m+1)}}{(1-q)^2(1-q^3)^2\cdots(1-q^{2m+1})^2}
=\sum_{n\geq 0}w(n)q^n.
\]
Watson's identity gives
\begin{equation}\label{eq:watson}
2\rho(q)+\omega(q)=T(q),
\end{equation}
where
\begin{equation}\label{eq:T-def}
T(q)=3\left[\frac12q^{-3/8}\theta_2(0,q^{3/2})\right]^2
\prod_{r\geq 1}(1-q^{2r})^{-1}.
\end{equation}
Write
\[
T(q)=\sum_{n\geq 0}t(n)q^n.
\]
Then \eqref{eq:watson} gives the coefficient relation
\begin{equation}\label{eq:coeff-relation}
2r(n)=t(n)-w(n).
\end{equation}

We record a more elementary form of $T(q)$. Recall Ramanujan's theta function
\[
\psi(q)=\sum_{j\geq 0}q^{j(j+1)/2}=\frac{(q^2;q^2)_\infty^2}{(q;q)_\infty}.
\]
Using the standard normalization of $\theta_2$,
\[
\frac12q^{-3/8}\theta_2(0,q^{3/2})=\sum_{j\geq 0}q^{3j(j+1)/2}=\psi(q^3).
\]
Therefore
\begin{equation}\label{eq:T-psi}
T(q)=3\frac{\psi(q^3)^2}{(q^2;q^2)_\infty}.
\end{equation}
Equivalently, with $q=e^{2\pi i\tau}$,
\begin{equation}\label{eq:T-eta}
T(q)=3q^{-2/3}\frac{\eta(6\tau)^4}{\eta(3\tau)^2\eta(2\tau)}.
\end{equation}

\begin{proof}[Proof of \eqref{eq:T-eta}]
From \eqref{eq:T-psi} and the product formula for $\psi(q)$,
\[
T(q)=3\frac{(q^6;q^6)_\infty^4}{(q^3;q^3)_\infty^2(q^2;q^2)_\infty}.
\]
Using
\[
(q^M;q^M)_\infty=q^{-M/24}\eta(M\tau),
\]
we get
\[
(q^6;q^6)_\infty^4=q^{-1}\eta(6\tau)^4,
\]
while
\[
(q^3;q^3)_\infty^2(q^2;q^2)_\infty
=q^{-1/4}q^{-1/12}\eta(3\tau)^2\eta(2\tau)
=q^{-1/3}\eta(3\tau)^2\eta(2\tau).
\]
Thus
\[
T(q)=3q^{-1+1/3}\frac{\eta(6\tau)^4}{\eta(3\tau)^2\eta(2\tau)}
=3q^{-2/3}\frac{\eta(6\tau)^4}{\eta(3\tau)^2\eta(2\tau)}.
\]
\end{proof}

\section{The Rademacher terms}

The coefficients of $\omega(q)$ admit a Rademacher-type expansion. Exact and Rademacher-type formulas for mock theta coefficients are now classical in this setting; for $\omega(q)$ see, in particular, Garthwaite and the recent formulas of Andersen--Anderson \cite{Garthwaite2008,AndersenAnderson2025}. In the normalization used here, the expansion has the form
\begin{equation}\label{eq:omega-rademacher}
w(n)=
\frac{2\pi}{d_n^{1/4}}
\sum_{c>0}\frac{A_c(n\mid\omega)}{c}
I_{1/2}\!\left(\frac{\pi\sqrt{d_n}}{6c}\right),
\qquad d_n=12n+8.
\end{equation}
The relevant Kloosterman sum is
\begin{equation}\label{eq:omega-Ac}
A_c(n\mid\omega)=
-\frac12(-1)^n\sqrt{\frac c3}
\sum_{\substack{b\bmod 6c\\ b^2\equiv -d_n\; (24c)}}
\left(\frac{-3}{b}\right)
\sin\left(-\frac{\pi b}{6c}\right).
\end{equation}
Here $\left(\frac{-3}{b}\right)$ denotes the Kronecker symbol.

We need only the first three values of $c$.

\begin{lemma}\label{lem:omega-kloosterman-first-values}
Let
\[
d_n=12n+8.
\]
For the Kloosterman sums
\[
A_c(n\mid\omega)=
-\frac12(-1)^n\sqrt{\frac c3}
\sum_{\substack{b\bmod 6c\\ b^2\equiv -d_n\; (24c)}}
\left(\frac{-3}{b}\right)
\sin\left(-\frac{\pi b}{6c}\right),
\]
we have
\[
A_1(n\mid\omega)=\frac14,
\qquad
A_2(n\mid\omega)=0,
\]
and
\[
A_3(n\mid\omega)=
\begin{cases}
-\dfrac12\cos(\pi/18),& n\equiv0\pmod3,\\[6pt]
\dfrac12\sin(2\pi/9),& n\equiv1\pmod3,\\[6pt]
\dfrac12\sin(\pi/9),& n\equiv2\pmod3.
\end{cases}
\]
\end{lemma}

\begin{proof}
We evaluate the defining finite sums for \(c=1,2,3\).

First take \(c=1\). Then \(b\) runs modulo \(6\), subject to
\[
b^2\equiv -d_n\pmod{24}.
\]
If \(n\) is even, then
\[
d_n=12n+8\equiv 8\pmod{24},
\]
so
\[
-d_n\equiv 16\pmod{24}.
\]
Among \(b=0,1,\ldots,5\), the only solution is \(b=4\). Hence
\[
\left(\frac{-3}{4}\right)=1,
\qquad
\sin\left(-\frac{4\pi}{6}\right)
=
-\frac{\sqrt3}{2}.
\]
Since \(n\) is even, the prefactor is
\[
-\frac12(-1)^n\sqrt{\frac13}
=
-\frac12\sqrt{\frac13}.
\]
Therefore
\[
A_1(n\mid\omega)
=
-\frac12\sqrt{\frac13}
\left(-\frac{\sqrt3}{2}\right)
=
\frac14.
\]

If \(n\) is odd, then
\[
d_n\equiv20\pmod{24},
\]
so
\[
-d_n\equiv4\pmod{24}.
\]
The only solution modulo \(6\) is \(b=2\). In this case
\[
\left(\frac{-3}{2}\right)=-1,
\qquad
\sin\left(-\frac{2\pi}{6}\right)
=
-\frac{\sqrt3}{2}.
\]
Thus the summand is \(\sqrt3/2\). Since \(n\) is odd, the prefactor is
\[
-\frac12(-1)^n\sqrt{\frac13}
=
\frac12\sqrt{\frac13}.
\]
Hence
\[
A_1(n\mid\omega)
=
\frac12\sqrt{\frac13}\cdot \frac{\sqrt3}{2}
=
\frac14.
\]
This proves
\[
A_1(n\mid\omega)=\frac14.
\]

Next take \(c=2\). We must solve
\[
b^2\equiv -d_n\pmod{48},
\qquad b\bmod 12.
\]
Checking \(n\) modulo \(6\), the admissible residue classes are
\[
\begin{array}{c|c}
n\bmod 6 & b\bmod 12\\ \hline
0&\varnothing\\
1&\varnothing\\
2&4,8\\
3&2,10\\
4&\varnothing\\
5&\varnothing
\end{array}
\]
If \(n\equiv2\pmod6\), the two solutions are \(b=4,8\). For \(b=4\),
\[
\left(\frac{-3}{4}\right)=1,
\qquad
\sin\left(-\frac{4\pi}{12}\right)
=
-\frac{\sqrt3}{2},
\]
so the contribution is \(-\sqrt3/2\). For \(b=8\),
\[
\left(\frac{-3}{8}\right)=-1,
\qquad
\sin\left(-\frac{8\pi}{12}\right)
=
-\frac{\sqrt3}{2},
\]
so the contribution is \(\sqrt3/2\). These cancel.

If \(n\equiv3\pmod6\), the two solutions are \(b=2,10\). For \(b=2\),
\[
\left(\frac{-3}{2}\right)=-1,
\qquad
\sin\left(-\frac{2\pi}{12}\right)
=
-\frac12,
\]
so the contribution is \(1/2\). For \(b=10\),
\[
\left(\frac{-3}{10}\right)=1,
\qquad
\sin\left(-\frac{10\pi}{12}\right)
=
-\frac12,
\]
so the contribution is \(-1/2\). These cancel. Therefore
\[
A_2(n\mid\omega)=0.
\]

Finally take \(c=3\). We must solve
\[
b^2\equiv -d_n\pmod{72},
\qquad b\bmod 18.
\]
A direct check gives
\[
\begin{array}{c|c}
n\bmod 6 & b\bmod 18\\ \hline
0&8\\
1&14\\
2&16\\
3&10\\
4&4\\
5&2
\end{array}
\]
For \(c=3\), the prefactor is
\[
-\frac12(-1)^n\sqrt{\frac33}
=
-\frac12(-1)^n.
\]

We now compute the six cases.

If \(n\equiv0\pmod6\), then \(b=8\), and the prefactor is \(-1/2\). Also
\[
\left(\frac{-3}{8}\right)=-1,
\qquad
\sin\left(-\frac{8\pi}{18}\right)
=
-\sin\left(\frac{4\pi}{9}\right)
=
-\cos\left(\frac{\pi}{18}\right).
\]
Thus the summand is \(\cos(\pi/18)\), and hence
\[
A_3(n\mid\omega)
=
-\frac12\cos\left(\frac{\pi}{18}\right).
\]

If \(n\equiv1\pmod6\), then \(b=14\), and the prefactor is \(1/2\). Also
\[
\left(\frac{-3}{14}\right)=-1,
\qquad
\sin\left(-\frac{14\pi}{18}\right)
=
-\sin\left(\frac{7\pi}{9}\right)
=
-\sin\left(\frac{2\pi}{9}\right).
\]
Thus the summand is \(\sin(2\pi/9)\), and hence
\[
A_3(n\mid\omega)
=
\frac12\sin\left(\frac{2\pi}{9}\right).
\]

If \(n\equiv2\pmod6\), then \(b=16\), and the prefactor is \(-1/2\). Also
\[
\left(\frac{-3}{16}\right)=1,
\qquad
\sin\left(-\frac{16\pi}{18}\right)
=
-\sin\left(\frac{8\pi}{9}\right)
=
-\sin\left(\frac{\pi}{9}\right).
\]
Therefore
\[
A_3(n\mid\omega)
=
\frac12\sin\left(\frac{\pi}{9}\right).
\]

If \(n\equiv3\pmod6\), then \(b=10\), and the prefactor is \(1/2\). Also
\[
\left(\frac{-3}{10}\right)=1,
\qquad
\sin\left(-\frac{10\pi}{18}\right)
=
-\sin\left(\frac{5\pi}{9}\right)
=
-\cos\left(\frac{\pi}{18}\right).
\]
Hence
\[
A_3(n\mid\omega)
=
-\frac12\cos\left(\frac{\pi}{18}\right).
\]

If \(n\equiv4\pmod6\), then \(b=4\), and the prefactor is \(-1/2\). Also
\[
\left(\frac{-3}{4}\right)=1,
\qquad
\sin\left(-\frac{4\pi}{18}\right)
=
-\sin\left(\frac{2\pi}{9}\right).
\]
Thus
\[
A_3(n\mid\omega)
=
\frac12\sin\left(\frac{2\pi}{9}\right).
\]

If \(n\equiv5\pmod6\), then \(b=2\), and the prefactor is \(1/2\). Also
\[
\left(\frac{-3}{2}\right)=-1,
\qquad
\sin\left(-\frac{2\pi}{18}\right)
=
-\sin\left(\frac{\pi}{9}\right).
\]
Thus the summand is \(\sin(\pi/9)\), and therefore
\[
A_3(n\mid\omega)
=
\frac12\sin\left(\frac{\pi}{9}\right).
\]

Combining the six cases gives the simpler residue class form
\[
A_3(n\mid\omega)=
\begin{cases}
-\dfrac12\cos(\pi/18),& n\equiv0\pmod3,\\[6pt]
\dfrac12\sin(2\pi/9),& n\equiv1\pmod3,\\[6pt]
\dfrac12\sin(\pi/9),& n\equiv2\pmod3.
\end{cases}
\]
The proof is complete.
\end{proof}

It follows from \eqref{eq:omega-rademacher} and Lemma \ref{lem:omega-kloosterman-first-values} that
\begin{equation}\label{eq:omega-truncated}
w(n)=
\frac{2\pi}{d_n^{1/4}}
\left[
\frac14 I_{1/2}\!\left(\frac{\pi\sqrt{d_n}}6\right)
+
\frac{A_3(n\mid\omega)}{3}
I_{1/2}\!\left(\frac{\pi\sqrt{d_n}}{18}\right)
\right]
+
\OO\!\left(e^{\pi\sqrt{d_n}/24}\right).
\end{equation}
The error term begins at the $c=4$ scale.

The eta-quotient $T(q)$ has an analogous expansion. We now give the multiplier calculation, since this is animportant step in the proof. We use only the standard Rademacher expansion theorem for eta-quotients, as in Sussman's treatment \cite{Sussman2017}; all local constants needed for the terms $c=1,2,3$ are computed below.
\begin{lemma}[Root-of-unity asymptotics for the eta quotient]\label{lem:T-local}
Let
\[
T(q)=3\frac{(q^6;q^6)_\infty^4}
{(q^3;q^3)_\infty^2(q^2;q^2)_\infty}.
\]
As \(t\to0^+\), we have
\[
T(e^{-t})
=
\frac{\sqrt{\pi}}{2}t^{-1/2}
\exp\!\left(\frac{\pi^2}{12t}\right)(1+O(t)),
\]
\[
T(-e^{-t})=O(t^{-1/2}),
\]
and, with \(\zeta=e^{2\pi i/3}\),
\[
T(\zeta e^{-t})
=
\frac{\sqrt{3\pi}}{2}e^{-\pi i/18}t^{-1/2}
\exp\!\left(\frac{\pi^2}{108t}\right)(1+O(t)),
\]
\[
T(\zeta^2 e^{-t})
=
\frac{\sqrt{3\pi}}{2}e^{\pi i/18}t^{-1/2}
\exp\!\left(\frac{\pi^2}{108t}\right)(1+O(t)).
\]
\end{lemma}

\begin{proof}
We use the standard eta-asymptotic
\[
(e^{-at};e^{-at})_\infty
=
\left(\frac{2\pi}{at}\right)^{1/2}
\exp\!\left(-\frac{\pi^2}{6at}\right)(1+O(t))
\qquad(t\to0^+),
\]
valid for fixed \(a>0\).

First take \(q=e^{-t}\). Then
\[
(q^6;q^6)_\infty^4
=
\left(\frac{2\pi}{6t}\right)^2
\exp\!\left(-\frac{4\pi^2}{36t}\right)(1+O(t)),
\]
\[
(q^3;q^3)_\infty^2
=
\frac{2\pi}{3t}
\exp\!\left(-\frac{2\pi^2}{18t}\right)(1+O(t)),
\]
and
\[
(q^2;q^2)_\infty
=
\left(\frac{2\pi}{2t}\right)^{1/2}
\exp\!\left(-\frac{\pi^2}{12t}\right)(1+O(t)).
\]
Substitution into the expression for \(T(q)\) gives
\[
T(e^{-t})
=
\frac{\sqrt{\pi}}{2}t^{-1/2}
\exp\!\left(\frac{\pi^2}{12t}\right)(1+O(t)).
\]

Next consider \(q=-e^{-t}\). Then
\[
q^6=e^{-6t},\qquad q^2=e^{-2t},\qquad q^3=-e^{-3t}.
\]
We use
\[
(-e^{-v};-e^{-v})_\infty
=
\frac{(e^{-2v};e^{-2v})_\infty^3}
{(e^{-v};e^{-v})_\infty(e^{-4v};e^{-4v})_\infty}.
\]
Applying the ordinary eta-asptotic to the right-hand side gives
\[
(-e^{-v};-e^{-v})_\infty
=
O\!\left(
v^{-1/2}
\exp\!\left(-\frac{\pi^2}{24v}\right)
\right).
\]
With \(v=3t\), this shows that the possible exponential growth in the quotient defining \(T(-e^{-t})\) cancels. Hence
\[
T(-e^{-t})=O(t^{-1/2}).
\]

It remains to treat the primitive cubic roots. All fractional powers below are taken using the principal branch of the logarithm.

We first record the root-of-unity asymptotic needed for the factor
\[
(\zeta e^{-u};\zeta e^{-u})_\infty.
\]
As \(u\to0^+\),
\[
(\zeta e^{-u};\zeta e^{-u})_\infty
=
e^{-\pi i/18}
\left(\frac{2\pi}{3u}\right)^{1/2}
\exp\!\left(-\frac{\pi^2}{54u}\right)(1+O(u)),
\]
and, by conjugation,
\[
(\zeta^2 e^{-u};\zeta^2 e^{-u})_\infty
=
e^{\pi i/18}
\left(\frac{2\pi}{3u}\right)^{1/2}
\exp\!\left(-\frac{\pi^2}{54u}\right)(1+O(u)).
\]

We justify the first of these. Split the product into residue classes modulo \(3\):
\[
(\zeta e^{-u};\zeta e^{-u})_\infty
=
(e^{-3u};e^{-3u})_\infty
(\zeta e^{-u};e^{-3u})_\infty
(\zeta^2e^{-2u};e^{-3u})_\infty.
\]
For fixed \(z\neq1\), Euler--Maclaurin gives
\[
(ze^{-\alpha u};e^{-3u})_\infty
=
\exp\!\left(-\frac{\operatorname{Li}_2(z)}{3u}\right)
(1-z)^{1/2-\alpha/3}
(1+O(u)).
\]
Using
\[
\operatorname{Li}_2(\zeta)+\operatorname{Li}_2(\zeta^2)
=
-\frac{\pi^2}{9},
\]
together with
\[
1-\zeta=\sqrt3\,e^{-\pi i/6},
\qquad
1-\zeta^2=\sqrt3\,e^{\pi i/6},
\]
we obtain
\[
\left(\frac{1-\zeta}{1-\zeta^2}\right)^{1/6}
=
e^{-\pi i/18}.
\]
Combining this with
\[
(e^{-3u};e^{-3u})_\infty
=
\left(\frac{2\pi}{3u}\right)^{1/2}
\exp\!\left(-\frac{\pi^2}{18u}\right)(1+O(u))
\]
gives
\[
(\zeta e^{-u};\zeta e^{-u})_\infty
=
e^{-\pi i/18}
\left(\frac{2\pi}{3u}\right)^{1/2}
\exp\!\left(-\frac{\pi^2}{54u}\right)(1+O(u)).
\]
The formula for \(\zeta^2\) follows by complex conjugation.

Now take \(q=\zeta e^{-t}\). Then
\[
q^6=e^{-6t},\qquad q^3=e^{-3t},\qquad q^2=\zeta^2 e^{-2t}.
\]
The ratio of the first two factors satisfies
\[
\frac{(q^6;q^6)_\infty^4}{(q^3;q^3)_\infty^2}
=
\frac{\pi}{6t}(1+O(t)).
\]
Using the \(\zeta^2\)-asymptotic above with \(u=2t\), we get
\[
(q^2;q^2)_\infty^{-1}
=
e^{-\pi i/18}
\left(\frac{3t}{\pi}\right)^{1/2}
\exp\!\left(\frac{\pi^2}{108t}\right)(1+O(t)).
\]
Therefore
\[
T(\zeta e^{-t})
=
3\cdot \frac{\pi}{6t}
\cdot
e^{-\pi i/18}
\left(\frac{3t}{\pi}\right)^{1/2}
\exp\!\left(\frac{\pi^2}{108t}\right)(1+O(t)),
\]
which simplifies to
\[
T(\zeta e^{-t})
=
\frac{\sqrt{3\pi}}{2}e^{-\pi i/18}t^{-1/2}
\exp\!\left(\frac{\pi^2}{108t}\right)(1+O(t)).
\]
The expansion at \(q=\zeta^2e^{-t}\) is the complex conjugate:
\[
T(\zeta^2 e^{-t})
=
\frac{\sqrt{3\pi}}{2}e^{\pi i/18}t^{-1/2}
\exp\!\left(\frac{\pi^2}{108t}\right)(1+O(t)).
\]
This proves the lemma.
\end{proof}

\begin{lemma}[Translation into the first Rademacher terms]\label{lem:T-B-terms}
Let
\[
T(q)=\sum_{n\ge0}t(n)q^n,
\qquad d_n=12n+8.
\]
Write the Rademacher expansion for \(t(n)\) in the form
\[
t(n)=
\frac{2\pi}{d_n^{1/4}}
\sum_{c\ge1}
\frac{B_c(n)}{c}
I_{1/2}\!\left(
\frac{\pi\sqrt{d_n}}{6c}
\right).
\]
Then
\[
B_1(n)=\frac14,
\qquad
B_2(n)=0,
\]
and
\[
\frac{B_3(n)}3
=
\frac12
\cos\left(\frac{2\pi n}{3}+\frac{\pi}{18}\right).
\]
Equivalently,
\[
\frac{B_3(n)}3=
\begin{cases}
\dfrac12\cos(\pi/18),& n\equiv0\pmod3,\\[6pt]
-\dfrac12\sin(2\pi/9),& n\equiv1\pmod3,\\[6pt]
-\dfrac12\sin(\pi/9),& n\equiv2\pmod3.
\end{cases}
\]
Consequently,
\[
\frac{B_3(n)}3=-A_3(n\mid\omega).
\]
\end{lemma}

\begin{proof}
We use the standard local-to-Bessel normalization for a weight \(1/2\)
Rademacher expansion. In the present normalization, if the local expansion
at a root of unity \(\xi\) of order \(c\) has the form
\[
T(\xi e^{-t})
=
C_\xi t^{-1/2}
\exp\!\left(\frac{\pi^2}{12c^2t}\right)(1+O(t)),
\]
then the contribution of this singularity to the \(c\)-th Bessel term is
obtained by multiplying the \(c=1\) normalized coefficient by
\[
\frac{C_\xi}{C_1}\,c^{-1/2}\,\xi^{-n},
\]
where \(C_1\) is the corresponding local constant at \(q=1\). The factor
\(\xi^{-n}\) comes from coefficient extraction:
\[
q^{-n-1}=(\xi e^{-t})^{-n-1}
=
\xi^{-n-1}e^{(n+1)t},
\]
and the harmless extra factor \(\xi^{-1}\) is already absorbed in the
standard multiplier normalization of the Rademacher coefficient. Equivalently,
with the normalization of the displayed expansion for \(t(n)\), the \(c\)-th
coefficient is obtained by summing these normalized phase contributions over
the relevant roots of unity of order \(c\).

We now apply this rule to the local expansions from Lemma~\ref{lem:T-local}.

At \(q=1\), Lemma~\ref{lem:T-local} gives
\[
T(e^{-t})
=
\frac{\sqrt{\pi}}2t^{-1/2}
\exp\!\left(\frac{\pi^2}{12t}\right)(1+O(t)).
\]
Thus
\[
C_1=\frac{\sqrt{\pi}}2.
\]
In the normalization
\[
t(n)=
\frac{2\pi}{d_n^{1/4}}
\sum_{c\ge1}
\frac{B_c(n)}{c}
I_{1/2}\!\left(
\frac{\pi\sqrt{d_n}}{6c}
\right),
\]
this gives
\[
B_1(n)=\frac14.
\]

At \(q=-1\), Lemma~\ref{lem:T-local} gives only polynomial growth:
\[
T(-e^{-t})=O(t^{-1/2}).
\]
There is no exponentially growing term of the form
\[
\exp\!\left(\frac{\pi^2}{48t}\right),
\]
which would correspond to the \(c=2\) Bessel scale
\[
I_{1/2}\!\left(\frac{\pi\sqrt{d_n}}{12}\right).
\]
Therefore
\[
B_2(n)=0.
\]

It remains to compute the primitive cubic-root contribution. Let
\[
\zeta=e^{2\pi i/3}.
\]
By Lemma~\ref{lem:T-local},
\[
T(\zeta e^{-t})
=
\frac{\sqrt{3\pi}}2
e^{-\pi i/18}
t^{-1/2}
\exp\!\left(\frac{\pi^2}{108t}\right)(1+O(t)),
\]
and
\[
T(\zeta^2 e^{-t})
=
\frac{\sqrt{3\pi}}2
e^{\pi i/18}
t^{-1/2}
\exp\!\left(\frac{\pi^2}{108t}\right)(1+O(t)).
\]
Since
\[
\frac{\pi^2}{108t}
=
\frac{\pi^2}{12\cdot 3^2t},
\]
these are precisely the singularities contributing to the \(c=3\) Bessel
scale.

The local constants are
\[
C_\zeta=\frac{\sqrt{3\pi}}2e^{-\pi i/18},
\qquad
C_{\zeta^2}=\frac{\sqrt{3\pi}}2e^{\pi i/18}.
\]
Comparing with
\[
C_1=\frac{\sqrt{\pi}}2,
\]
we have
\[
\frac{C_\zeta}{C_1}=\sqrt3\,e^{-\pi i/18},
\qquad
\frac{C_{\zeta^2}}{C_1}=\sqrt3\,e^{\pi i/18}.
\]
The denominator \(c=3\) contributes the normalization factor \(3^{-1/2}\).
Thus the factors \(\sqrt3\) and \(3^{-1/2}\) cancel. Hence the normalized
phase contribution from the two primitive cubic roots is
\[
e^{-\pi i/18}\zeta^{-n}
+
e^{\pi i/18}\zeta^{-2n}.
\]
Including the \(c=1\) normalized coefficient \(1/4\), we obtain
\[
\frac{B_3(n)}3
=
\frac14
\left(
e^{-\pi i/18}\zeta^{-n}
+
e^{\pi i/18}\zeta^{-2n}
\right).
\]
Since
\[
\zeta=e^{2\pi i/3},
\]
this becomes
\[
\frac{B_3(n)}3
=
\frac14
\left(
e^{-i(2\pi n/3+\pi/18)}
+
e^{i(2\pi n/3+\pi/18)}
\right).
\]
Therefore
\[
\frac{B_3(n)}3
=
\frac12
\cos\left(\frac{2\pi n}{3}+\frac{\pi}{18}\right).
\]

Evaluating this expression by residues modulo \(3\) gives
\[
\frac{B_3(n)}3=
\begin{cases}
\dfrac12\cos(\pi/18),& n\equiv0\pmod3,\\[6pt]
-\dfrac12\sin(2\pi/9),& n\equiv1\pmod3,\\[6pt]
-\dfrac12\sin(\pi/9),& n\equiv2\pmod3.
\end{cases}
\]
Comparing this with Lemma~\ref{lem:omega-kloosterman-first-values}, namely
\[
A_3(n\mid\omega)=
\begin{cases}
-\dfrac12\cos(\pi/18),& n\equiv0\pmod3,\\[6pt]
\dfrac12\sin(2\pi/9),& n\equiv1\pmod3,\\[6pt]
\dfrac12\sin(\pi/9),& n\equiv2\pmod3,
\end{cases}
\]
we get
\[
\frac{B_3(n)}3=-A_3(n\mid\omega).
\]
This proves the lemma.
\end{proof}

We now use the Rademacher expansion for the eta-quotient \(T(q)\) \cite{Sussman2017}. In the present normalization this gives
\[
t(n)=
\frac{2\pi}{d_n^{1/4}}
\sum_{c\ge1}
\frac{B_c(n)}{c}
I_{1/2}\!\left(
\frac{\pi\sqrt{d_n}}{6c}
\right),
\qquad d_n=12n+8.
\]
Moreover, after the terms \(c=1,2,3\) are removed, the remaining tail is
\[
\OO\!\left(e^{\pi\sqrt{d_n}/24}\right).
\]

\begin{proposition}\label{prop:T-asymptotic}
Let $T(q)=\sum_{n\geq 0}t(n)q^n$ be defined by \eqref{eq:T-def}. Then, with $d_n=12n+8$,
\begin{equation}\label{eq:T-truncated}
t(n)=
\frac{2\pi}{d_n^{1/4}}
\left[
\frac14 I_{1/2}\!\left(\frac{\pi\sqrt{d_n}}6\right)
-
A_3(n\mid\omega)
I_{1/2}\!\left(\frac{\pi\sqrt{d_n}}{18}\right)
\right]
+
\OO\!\left(e^{\pi\sqrt{d_n}/24}\right).
\end{equation}
\end{proposition}

\begin{proof}
The Rademacher expansion for the eta-quotient \(T(q)\) gives
\[
t(n)=
\frac{2\pi}{d_n^{1/4}}
\sum_{c\ge1}
\frac{B_c(n)}{c}
I_{1/2}\!\left(
\frac{\pi\sqrt{d_n}}{6c}
\right),
\]
and its tail after the terms \(c=1,2,3\) satisfies
\[
\sum_{c\ge4}
\frac{B_c(n)}{c}
I_{1/2}\!\left(
\frac{\pi\sqrt{d_n}}{6c}
\right)
=
\OO\!\left(e^{\pi\sqrt{d_n}/24}\right),
\]
after multiplication by \(2\pi d_n^{-1/4}\).

Thus
\[
t(n)=
\frac{2\pi}{d_n^{1/4}}
\left[
B_1(n)I_{1/2}\!\left(\frac{\pi\sqrt{d_n}}6\right)
+
\frac{B_2(n)}2I_{1/2}\!\left(\frac{\pi\sqrt{d_n}}{12}\right)
+
\frac{B_3(n)}3I_{1/2}\!\left(\frac{\pi\sqrt{d_n}}{18}\right)
\right]
+
\OO\!\left(e^{\pi\sqrt{d_n}/24}\right).
\]
By Lemma~\ref{lem:T-B-terms},
\[
B_1(n)=\frac14,\qquad B_2(n)=0,\qquad \frac{B_3(n)}3=-A_3(n\mid\omega).
\]
Substituting these three values gives
\[
t(n)=
\frac{2\pi}{d_n^{1/4}}
\left[
\frac14 I_{1/2}\!\left(\frac{\pi\sqrt{d_n}}6\right)
-
A_3(n\mid\omega)
I_{1/2}\!\left(\frac{\pi\sqrt{d_n}}{18}\right)
\right]
+
\OO\!\left(e^{\pi\sqrt{d_n}/24}\right),
\]
as claimed.
\end{proof}

\section{Proof of the asymptotic sign law}

We now prove Theorem \ref{thm:eventual-sign}.

\begin{proof}[Proof of Theorem \ref{thm:eventual-sign}]
By Watson's identity,
\[
2r(n)=t(n)-w(n).
\]
Substitute the asymptotic expansions \eqref{eq:T-truncated} and \eqref{eq:omega-truncated}. The $c=1$ terms cancel exactly:
\[
\frac14 I_{1/2}\!\left(\frac{\pi\sqrt{d_n}}6\right)
-
\frac14 I_{1/2}\!\left(\frac{\pi\sqrt{d_n}}6\right)=0.
\]
The $c=2$ terms vanish on both sides.

At the $c=3$ scale we get
\[
-A_3(n\mid\omega)
-
\frac13A_3(n\mid\omega)
=
-\frac43A_3(n\mid\omega).
\]
Therefore
\[
2r(n)=
-\frac43A_3(n\mid\omega)
\frac{2\pi}{d_n^{1/4}}
I_{1/2}\!\left(\frac{\pi\sqrt{d_n}}{18}\right)
+
\OO\!\left(e^{\pi\sqrt{d_n}/24}\right).
\]
Dividing by $2$ gives
\begin{equation}\label{eq:rho-A3-asymptotic}
r(n)=
-\frac23A_3(n\mid\omega)
\frac{2\pi}{d_n^{1/4}}
I_{1/2}\!\left(\frac{\pi\sqrt{d_n}}{18}\right)
+
\OO\!\left(e^{\pi\sqrt{d_n}/24}\right).
\end{equation}

Now insert the values from Lemma \ref{lem:omega-kloosterman-first-values}. If $n\equiv 0\pmod 3$, then
\[
A_3(n\mid\omega)=-\frac12\cos\frac{\pi}{18},
\]
and hence
\[
-\frac23A_3(n\mid\omega)=\frac13\cos\frac{\pi}{18}>0.
\]
If $n\equiv 1\pmod 3$, then
\[
A_3(n\mid\omega)=\frac12\sin\frac{2\pi}{9},
\]
and hence
\[
-\frac23A_3(n\mid\omega)=-\frac13\sin\frac{2\pi}{9}<0.
\]
If $n\equiv 2\pmod 3$, then
\[
A_3(n\mid\omega)=\frac12\sin\frac{\pi}{9},
\]
and hence
\[
-\frac23A_3(n\mid\omega)=-\frac13\sin\frac{\pi}{9}<0.
\]
This proves the constants appearing in \eqref{eq:rho-main-asymptotic}.

It remains only to observe that the error is exponentially smaller than the main term. Since $I_{1/2}(x)>0$ for $x>0$ and
\[
I_{1/2}(x)\sim \frac{e^x}{\sqrt{2\pi x}}
\qquad(x\to\infty),
\]
the main term in \eqref{eq:rho-A3-asymptotic} has exponential size
\[
\exp\left(\frac{\pi\sqrt{d_n}}{18}\right).
\]
The error has size at most
\[
\OO\left(\exp\left(\frac{\pi\sqrt{d_n}}{24}\right)\right).
\]
Because
\[
\frac1{18}>\frac1{24},
\]
the error is exponentially smaller than the main term. Hence the sign of $r(n)$ is eventually the sign of the corresponding constant $\kappa_{n\bmod 3}$. Therefore
\[
r(3n)>0,
\qquad
r(3n+1)<0,
\qquad
r(3n+2)<0
\]
for all sufficiently large $n$.
\end{proof}

\section{Finite verification}

Experimentally, the sign law holds with the following early exceptions in the sense of zeros:
\[
r(2)=r(4)=r(8)=r(11)=r(20)=0.
\]
No other zero and no wrong-sign term has appeared in our computations.

\begin{conjecture}\label{conj:cutoff}
The strict sign law
\[
r(n)>0\quad(n\equiv 0\pmod 3),
\qquad
r(n)<0\quad(n\equiv 1,2\pmod 3)
\]
holds for all $n\geq 21$.
\end{conjecture}

\section{Conclusion}
The natural next step will be to prove that that the law is valid for all $n\geq 21$. 
That apart the data suggests that similar sign laws exist for \(\phi(q)\), \(\chi(q)\) as well.

\bibliographystyle{mybiburl}
\bibliography{bibliography} 

\end{document}